\documentclass{amsart}
\usepackage{amsmath,amsthm,amssymb,enumerate, comment, color}
\newcommand{\Aut}[1]{\operatorname{Aut}{(#1)}}

\newcommand{\bs}{\backslash}

\newcommand{\Fbdry}{\partial _F G}
\renewcommand{\phi}{\varphi}
\newcommand{\supp}[1]{\operatorname{supp}{#1}}
\newcommand{\Prob}[1]{\operatorname{Prob}{#1}}
\newcommand{\Sub}[1]{\operatorname{Sub}{(#1)}}
\newcommand{\Ad}[1]{\operatorname{Ad}{#1}}
\DeclareMathOperator\Span{span}
\newtheorem{theorem}{Theorem}[section]
\newtheorem{lemma}[theorem]{Lemma}
\newtheorem{proposition}[theorem]{Proposition}
\newtheorem{corollary}[theorem]{Corollary}

\newtheorem*{nntheorem}{Theorem}
\newtheorem{maintheorem}{Theorem}

\theoremstyle{remark}

\theoremstyle{definition}
\newtheorem{example}[theorem]{Example}

\title[Simplicity of crossed products]{Simplicity of crossed products of the actions of totally disconnected locally compact groups on their boundaries}
\author{Ryoya Arimoto}
\address{RIMS, Kyoto University, \mbox{606-8502} Japan}
\email{arimoto@kurims.kyoto-u.ac.jp}


\begin{document}

\begin{abstract}
We prove that if a totally disconnected locally compact group admits a topologically free boundary, then the reduced crossed product of continuous functions on its Furstenberg boundary by the group is simple.
We also prove a partial converse of this result.
\end{abstract}

\maketitle

\section{Introduction}
One of the fundamental constructions of operator algebras is the construction from groups or group actions and this gives rise to a natural question of which group or group action produces simple operator algebras.
In this paper, we will restrict our attention to the question for the reduced group $\mathrm{C}^*$-algebras and the reduced crossed products.
The reduced group $\mathrm{C}^*$-algebra for a locally compact group $G$ is a norm closure of $\lambda (C_c (G))$ in $B(L^2(G))$ where $\lambda \colon C_c(G) \to B(L^2(G))$ is the left regular representation; $\lambda (f) \xi := f * \xi$ for $f \in C_c(G)$ and $\xi \in L^2(G)$.
For the definition of the reduced crossed products, see Section \ref{Crossedproducts}.

For discrete groups, this question has long been studied.
The first example of $\mathrm{C}^*$-simple groups, i.e., a group whose reduced group $\mathrm{C}^*$-algebra is simple, was found by Powers in 1975 (\cite{Powers}); his result states that the free group of rank two is $\mathrm{C}^*$-simple.
Since then, many groups, including the Baumslag--Solitar groups (\cite{delaHarpePreaux}), were shown to be $\mathrm{C}^*$-simple using Powers' argument.
Finally, in 2017, Kalantar and Kennedy proved that a discrete group is $\mathrm{C}^*$-simple if and only if the group acts on its Furstenberg boundary topologically freely (\cite{KalantarKennedy}).
Breuillard, Kalantar, Kennedy, and Ozawa found that the extremal disconnectedness of the Furstenberg boundary of a discrete group ensures that topological freeness is equivalent to freeness for the action of a discrete group on its Furstenberg boundary (\cite{BreuillardKalantarKennedyOzawa}).

As for the reduced crossed products, in 1990, Kawamura and Tomiyama showed that for an amenable discrete group $\Gamma$ with an action on a compact Hausdorff space $X$, the reduced crossed product $C(X) \rtimes _r \Gamma$ is simple if and only if the action is minimal and topologically free (\cite{KawamuraTomiyama}).
In 1994, Archbold and Spielberg showed that for a discrete group $\Gamma$ with an action on a $\mathrm{C}^*$-algebra $A$, if the action is minimal and topologically free (both notions are adjusted to a non-commutative setting), then $A \rtimes _r \Gamma$ is simple (\cite{ArchboldSpielberg}).
They also extended the result of Kawamura and Tomiyama to amenable actions with simpler proof.
Kalantar and Kennedy showed that for a discrete group $\Gamma$, the simplicity of $C(\partial _F \Gamma) \rtimes _r \Gamma$ implies that the action $\Gamma \curvearrowright \partial _F \Gamma$ on the Furstenberg boundary is topologically free (\cite{KalantarKennedy}).

In conclusion, the following theorem is known.

\begin{nntheorem}[{See \cite[Theorem 6.2.]{KalantarKennedy} and \cite[Theorem 3.1.]{BreuillardKalantarKennedyOzawa}}] \label{KK}
Let $\Gamma$ be a discrete group.
The following conditions are equivalent.
\begin{enumerate}[(i)]
\item The reduced group $\mathrm{C}^*$-algebra $\mathrm{C}^*_{\lambda} ( \Gamma )$ of $\Gamma$ is simple.
\item The reduced crossed product $C( \partial _F \Gamma) \rtimes _r \Gamma$ is simple. \label{A}
\item The action $\Gamma \curvearrowright \partial _F \Gamma$ is free. \label{B}
\item The action of $\Gamma$ on some $\Gamma$-boundary is topologically free. \label{C}
\end{enumerate}
\end{nntheorem}

Recently, some progress has been made in that question for locally compact groups.
Raum proved that any $\mathrm{C}^*$-simple group is totally disconnected (\cite{Raum}).
The existence of a non-discrete $\mathrm{C}^*$-simple group was confirmed by Suzuki (\cite{Suzuki}) and he also showed that the class of $\mathrm{C}^*$-simple groups is abundant (\cite{Suzuki2}).
Raum also proved that the Schlichting completion of the Baumslag--Solitar group is $\mathrm{C}^*$-simple, using Suzuki's result (\cite{Raumerra}).

Progress from another point of view has also been made.
In 2022, Caprace, Le Boudec, and Matte Bon proved that totally disconnected groups with a piecewise minimal-strongly-proximal action admit a topologically free boundary (\cite{CapraceLeBoudecMatteBon}).
In particular, it is shown that the Neretin groups admit a topologically free boundary (see Section \ref{Neretin} or \cite[Section 4]{CapraceLeBoudecMatteBon} for the Neretin groups).
Recently, Le Boudec and Tsankov showed that a locally compact group admits a topologically free boundary if and only if the group acts freely on its Furstenberg boundary (\cite{LeBoudecTsankov}).
Combining these two results (since the latter result was announced beforehand, this conclusion is written in \cite{CapraceLeBoudecMatteBon}), we obtain that the Neretin groups act freely on their Furstenberg boundary.
This is the first example of a non-discrete compactly generated simple locally compact group acting freely on its Furstenberg boundary.
It is still not known whether the Neretin groups are $\mathrm{C}^*$-simple or not.

Our main results are motivated by the work of Kawamura and Tomiyama (\cite{KawamuraTomiyama}) and the work of Archbold and Spielberg (\cite{ArchboldSpielberg}).

\begin{maintheorem}[See Corollary \ref{maincor}] \label{mainintro}
Let $G$ be a totally disconnected locally compact group and $\Fbdry$ be its Furstenberg boundary.
If $G$ admits a topologically free boundary, then $C(\Fbdry) \rtimes _r G$ is simple.
\end{maintheorem}

A converse of Theorem \ref{mainintro} holds under the assumption that the group is exact.

\begin{maintheorem}[See Corollary \ref{conversecor}]
Let $G$ be an exact totally disconnected locally compact group.
If $C( \partial _F G ) \rtimes _r G$ is simple, then the action $G \curvearrowright \partial _F G$ is free.
\end{maintheorem}

Theorem \ref{mainintro} applies to two examples; the Neretin groups $\mathcal{N}_{d,k}$ and the HNN extensions, especially the Schlichting completions of the Baumslag--Solitar groups.
In particular, our result states that $C(\partial _F \mathcal{N}_{d,k}) \rtimes _r \mathcal{N}_{d,k}$ is simple where $\mathcal{N}_{d,k}$ is the Neretin group.

\subsection*{Acknowledgements}
The author is grateful to Professor Narutaka Ozawa for his valuable comments.
This work was supported by Grant-in-Aid for JSPS Fellows Grant Number JP23KJ1186 and by JSPS KAKENHI Grant Numbers JP20H01806 and JP24K00527.

\section{Preliminaries}
\subsection{Totally disconnected locally compact groups}
In this paper, locally compact groups are assumed to be second countable.
A \textbf{totally disconnected locally compact group} is a topological group whose underlying topological space is totally disconnected and locally compact.
By van Dantzig's theorem (\cite{vanDantzig}), a locally compact group is totally disconnected if and only if it admits a neighborhood basis of identity consisting of compact open subgroups.

\subsection{Group actions and the Furstenberg boundary}\label{groupactions}
Let $G$ be a locally compact group.
For a compact Hausdorff space $X$, $\Prob X$ denotes the space of Radon probability measures on $X$ with weak-$*$ topology.
We identify elements of $X$ with the corresponding Dirac measures in $\Prob{X}$.
An action $G \curvearrowright X$ of $G$ on a compact Hausdorff space $X$ is said to be \textbf{minimal} if there is no non-trivial $G$-invariant closed subset.
It is said to be \textbf{strongly proximal} if $\overline{G. \mu}$ contains a element of $X$ for every $\mu \in \Prob{X}$.
It is said to be \textbf{topologically free} if the subset $\{ x \in X \mid G_x = \{ e \} \}$ is dense in $X$.

A compact $G$-space is called a \textbf{$G$-boundary} if the action is minimal and strongly proximal.
The \textbf{Furstenberg boundary} $\Fbdry$ of $G$ is the universal $G$-boundary, that is, for every $G$-boundary $X$, there is a (unique) $G$-quotient map $\Fbdry \twoheadrightarrow X$.
For any locally compact groups, such a boundary exists  (see \cite[Proposition 4.6.]{Furstenberg}).

\subsection{Crossed products}\label{Crossedproducts}
Let $G$ be a locally compact group with a fixed left Haar measure $\mu$, $\Delta$ be the modular function of $G$, and $X$ be a compact $G$-space.
Then $C(X)$ is a $G$-$\mathrm{C}^*$-algebra with the action $\alpha$ defined by $[ \alpha _g f ] (x) = f(g^{-1}.x)$.
Let $C_c (G , C(X))$ be the set of continuous functions $f$ from $G$ to $C(X)$ whose support $\overline{ \{ g \in G \mid f(g) \neq 0 \} }$ is compact.
This becomes a $*$-algebra by the following operations:
\begin{align*}
f_1 * f_2 (g) &:= \int _G f_1(h) \alpha _h (f_2  (h^{-1} g)) \, dh,  \ f_1, f_2 \in C_c (G,C( X )), \\
f^*(g) &:= \Delta(g^{-1}) \alpha _g (f(g^{-1})^*),  \ f \in C_c (G,C( X )).
\end{align*}
For a faithful representation $\pi \colon C(X) \hookrightarrow B(H)$, define a $*$-representation $\Pi \colon C_c (G,C(X)) \to B(L^2 (G) \otimes H)$ by 
\[
\Pi(f) \xi (g) = \int _G \pi (\alpha _{g^{-1}} (f(h))) \xi(h^{-1}g) \, dh,
\]
here we identify $L^2(G) \otimes H$ with $L^2(G,H)$ and $\xi \in L^2(G,H)$.
The \textbf{reduced crossed product} $C(X) \rtimes _r G$ of $C(X)$ by $G$ is defined to be $\overline{\Pi (C_c (G,C(X)))}^{\| \cdot \|} \subset B(L^2(G) \otimes H)$.
This definition does not depend on the choice of a faithful representation $\pi \colon C(X) \hookrightarrow B(H)$.

For a compact open subgroup $K$ of $G$, $p_K$ denotes a projection $\displaystyle \frac{1}{\mu (K)} \chi _K$ in $C(X) \rtimes _r G$.
A net $(p_K) _{K<G:\text{compact open subgroup}}$ of projections forms an approximate unit for $C(X) \rtimes _r G$.
Note that for a compact open subgroup $K < G$, there is a canonical embedding $C(X) \rtimes _r K \subset C(X) \rtimes _r G$.
For the $K$-invariant function $f$ on $X$, $[f]_K \in C(X) \rtimes _r K$ denotes a constant function on $K$ whose value is $\displaystyle \frac{1}{\mu (K)} f$.
There is a faithful conditional expectation $E_K \colon p_K (C(X) \rtimes _r G) p_K \to [C( K \bs X )] _K$ such that $E_K (p_K f p_K) = (p_K f p_K) |_{K}$ for $f \in C_c (G,C(X))$ (see \cite[Lemma 2.1.]{Suzuki2}).

\subsection{Amenable actions}
We review the definition and properties of amenable actions (see \cite{Delaroche} for details).
The action of a locally compact group $G$ on a compact Hausdorff space $X$ is said to be \textbf{amenable} if there is a net $(m_i) _i$ of continuous maps from $X$ to $\Prob{(G)}$ such that $\lim _i \| g.m_i^x - m_i^{gx} \| _1 = 0$ uniformly on compact subsets of $X \times G$ where $\Prob{(G)}$ is the set of probability measures on $G$ equipped with the weak-$*$ topology.
If the action $G \curvearrowright X$ is amenable, then the full crossed product $C(X) \rtimes G$ and the reduced crossed product $C(X) \rtimes _r G$ coincide (\cite[Theorem 5.3.]{Delaroche}).

A locally compact group is said to be \textbf{amenable at infinity} if it admits an amenable action on a compact space.
In \cite[Theorem A]{BrodzkiCaveLi}, this property is shown to be equivalent to exactness for second countable locally compact groups.
Thus, by \cite[Proposition 3.4.]{Delaroche}, exactness is equivalent to the amenability of the action $G \curvearrowright \beta ^{\textrm{lu}} G$ where $\beta ^{\textrm{lu}} G$ is the Gelfand spectrum of $C^{\textrm{lu}}_{\textrm{b}}(G)$, the $\mathrm{C}^*$-algebra of bounded left uniformly continuous functions on $G$.
Using $C^{\textrm{lu}}_{\textrm{b}}(G)$, instead of $\ell ^{\infty} (G)$, we have the following theorem by the same way as \cite[Theorem 4.5.]{KalantarKennedy}.

\begin{theorem}\label{Furstenbergamenable}
If a locally compact group $G$ is exact, then the action $G \curvearrowright \partial _F G$ of $G$ on its Furstenberg boundary is amenable.
\end{theorem}

\subsection{The Chabouty topology and the stabilizer map}
For a locally compact group $G$, $\Sub{G}$ denotes the set of closed subgroups of $G$ equipped with the Chabauty topology; this topology is defined by the basis 
\[
\{ H \in \Sub{G} \mid H \cap K = \emptyset, H \cap U_1 \neq \emptyset, \ldots , H \cap U_n \neq \emptyset  \}
\]
where $K \subset G$ is a compact subset and $U_1 , \ldots U_n \subset G$ are open subsets.
Given a compact $G$-space $X$, one can obtain the stabilizer map $X \ni x \mapsto G_x \in \Sub{G}$.
This map is not necessarily continuous.
However, the stabilizer map associated with the Furstenberg boundary is continuous (\cite[Corollary 1.3.]{LeBoudecTsankov}).

The following lemma follows from \cite[Appendix]{Glimm}.

\begin{lemma}\label{measure}
Let $G$ be a locally compact group and $X$ be a compact $G$-space.
Assume that the stabilizer map $X \ni x \mapsto G_x \in \Sub{G}$ is continuous.
Then there is a family of measures $\{ \mu _x \} _{x \in X}$ such that
\begin{itemize}
\item for each $x \in X$, $\mu _x$ is a left Haar measure on $G_x$,
\item for every $f \in C_c (G)$, a map $x \mapsto \int _{G_x} f(a) \, d\mu _x (a)$ is continuous.
\end{itemize}
\end{lemma}

\subsection{Neretin groups}\label{Neretin}
In this subsection, we briefly recall the definition of the Neretin groups.
We refer \cite[Section 4]{CapraceLeBoudecMatteBon} and references therein to the reader for the definition and properties of the Neretin groups.
Let $d , k \geq  2$ be integers and $T_{d , k}$ be a rooted tree 
such that the root has $k$ descendants and the others have $d$ descendants.
An \textbf{almost automorphism} of $T_{d , k}$ is a triple $(A, B, \phi)$
where $A, B \subset T_{d , k}$ are finite subtrees containing the root with $| \partial A| = | \partial B |$ and $\phi \colon T_{d , k} \setminus A \rightarrow T_{d , k} \setminus B$ is an isomorphism.
The \textbf{Neretin group} $\mathcal{N}_{d, k}$ is the quotient of the set of all almost automorphisms by the relation which identifies two almost automorphisms $(A_1, B_1, \phi_1), \, (A_2, B_2, \phi_2)$ if there exists a finite subtree $\tilde{A} \subset T_{d , k}$ containing the root such that $A_1 , \, A_2 \subset \tilde{A}$ and $\phi_1 | _{T_{d , k} \setminus \tilde{A}} = \phi_2 | _{T_{d , k} \setminus \tilde{A}}$.
One can easily check that $\mathcal{N}_{d, k}$ is a group.

The Neretin group $\mathcal{N}_{d, k}$ admits a totally disconnected locally compact group topology such that the inclusion map $\Aut{T_{d, k}} \hookrightarrow \mathcal{N}_{d, k}$ is continuous and open.

\subsection{HNN extensions}
We recall and fix notations of HNN extensions and its Bass--Serre trees.
Let $H$ be a locally compact group, $H_0$ be its open subgroup, and $\theta \colon H_0 \hookrightarrow H$ be an open, injective, continuous homomorphism.
Then the HNN extension $G = \textrm{HNN} ( H, H_0, \theta )$ corresponding $H, H_0 , \text{and } \theta$ is defined to be $\langle H , t \mid t h_0 t^{-1} = \theta (h_0) \text{ for } h_0 \in H_0 \rangle$.
If systems of representatives $S_{-1}$ and $S_1$ of $H/H_0$ and $H / \theta (H_0) $ respectively such that $e \in S_1$ and $e \in S_{-1}$ are fixed, then any element $g$ of $G$ is represented as the unique normal form $g = g_1 t^{\varepsilon _1} g_2 t^{\varepsilon _2} \cdots g_n t^{\varepsilon _n} g_{n+1}$ where
\begin{itemize}
\item $\varepsilon _i \in \{ \pm 1 \}$ and $g_{n+1} \in H$,
\item $g_i \in S_{\varepsilon _i}$,
\item If $g_i = 1$, then $\varepsilon _{i-1} = \varepsilon _i$.
\end{itemize}
This ensures that the canonical homomorphism $H \to G$ is an embedding.
There is a unique topology on $G$ such that the embedding $H \hookrightarrow G$ is continuous and open (\cite[Proposition 8.B.10.]{CornulierdelaHarpe}).
A Bass--Serre tree associated with $G$ is as follows; 
the vertex set is $G / H$, the edge set is $G / H_0$, and an edge $g H_0 \in G / H_0$ connects $g H$ and $g t^{-1} H$.
The HNN extension acts on its Bass--Serre tree by left translation.

\begin{example}
Let $m, n$ be integers with $2 \leq |m| \leq n$.
The Baumslag--Solitar group $\mathrm{BS}(m,n)$ is defined as $\langle a , t \mid t a^m t^{-1} = a^n \rangle$; an HNN extension corresponding to $\mathbb{Z} = \langle a \rangle , m \mathbb{Z} = \langle a^m \rangle, \text{ and } \theta ( a^{m} ) = a^n$.
Its Bass--Serre tree is $(|m|+n)$-regular tree $T_{|m|+n}$.
Let $\mathrm{G} (m, n)$ be a closure of $\mathrm{BS}(m,n)$ in $\Aut{T_{|m|+n}}$.
This group is also an HNN extension, corresponding to $\overline{\langle a \rangle}, \overline{\langle a^m \rangle}, \text{ and } \overline{\theta} (h_0) = t h_0 t^{-1} \text{ for } h_0 \in \overline{\langle a^m \rangle}$. 
Note that it is shown by Raum that this group is $\mathrm{C}^*$-simple (\cite{Raumerra}).
\end{example}

Let $G = \textrm{HNN}(H, H_0, \theta)$ be an HNN extension.
If $1<[H: H_0]<\infty$ and $1<[H: \theta(H_0)] < \infty$, then $\partial T$ is closed in $T \sqcup \partial T$ with the topology defined as in \cite[Section 5.2.]{BrownOzawa} and the action of $G$ on its Bass--Serre tree $T$ is minimal, i.e., there is no non-trivial $G$-invariant subtree, and of general type, i.e., there exist two hyperbolic elements in $G$ without common fixed points in $\partial T$ (see \cite[Proposition 19.]{delaHarpePreaux}).
Therefore, the induced action of $G$ on the boundary of the Bass--Serre tree $\partial T$ is a boundary action.

\section{Main theorems and examples}
In this section, we will prove the main theorems and see examples.

\begin{theorem}\label{main}
Let $G$ be a totally disconnected locally compact group and $X$ be a compact $G$-space.
Assume that the action $G \curvearrowright X$ is topologically free and there exists a compact open subgroup $K$ of $G$ such that the restricted action $K \curvearrowright X$ is free.
Then the action $G \curvearrowright X$ has the intersection property, that is, for every non-zero ideal $I$ in $C(X) \rtimes _r G$, there exists a non-zero $G$-invariant ideal $J$ in $C(X)$ such that $J \rtimes _r G \subset I$.
\end{theorem}

The following proposition asserts that closed ideals in the crossed product of totally disconnected compact groups can be described in terms of invariant open subsets if the action is free.

\begin{proposition}\label{ideal}
Let $K$ be a totally disconnected compact group and $X$ be a compact Hausdorff space on which $K$ acts freely.
Then every closed ideal of $C(X) \rtimes _r K$ is of the form $C_0(U) \rtimes _r K$ where $U$ is a $K$-invariant open subset of $X$.
\end{proposition}

\begin{proof}
Let $I$ be a closed ideal of $C(X) \rtimes _r K$.
For each compact open normal subgroup $L \lhd K$, $p_L (C(X) \rtimes _r K) p_L \cong C(L \backslash X) \rtimes _r K/L$.
Since $K/L$ is finite and the action $K/L \curvearrowright L \backslash X$ is free, by theorems by Kawamura--Tomiyama (\cite[Theorem 4.1.]{KawamuraTomiyama}) and Sierakowski (\cite[Theorem 1.13.]{Sierakowski}), $p_L I p_L$ is of the form $C_0 (U_L) \rtimes _r K/L$ where $U_L \subset L \backslash X$ is a $K/L$-invariant open subset.
Let $\tilde{U_L}$ be an inverse image of $U_L$ under the canonical quotient map $X \to L \backslash X$, which is $K$-invariant open subset.
Then $\tilde{U_L}$ does not depend on compact open normal subgroups $L$. 
Indeed, for compact open subgroups $L_1, L_2 \lhd K$ with $L_1 < L_2$, one has $L_2 \backslash \tilde{U_{L_1}} = U_{L_2}$ since $C_0(U_{L_2}) \rtimes _r K / L_2 = p_{L_2} I p_{L_2} = p_{L_2} ( p_{L_1} I p_{L_1} ) p_{L_2} = C_0 (L_2 \backslash \tilde{U_{L_1}}) \rtimes _r K / L_2$.
Hence, one has $\tilde{U_{L_1}} = \tilde{U_{L_2}}$.
For general compact open normal subgroups $L_1, L_2 \lhd K$, $\tilde{U_{L_1}} = \tilde{U_{L_1 \cap L_2}} = \tilde{U_{L_2}}$.
Therefore, we have $I = C_0 (\tilde{U}) \rtimes _r K$ where $\tilde{U} = \tilde{U_L}$ for some compact open normal subgroup $L \lhd K$.
\end{proof}

The proof of the next lemma is inspired by \cite{ExelLacaQuigg}.

\begin{lemma}\label{intersection}
Let $G$ be a totally disconnected locally compact group and $X$ be a compact $G$-space.
Assume that the action $G \curvearrowright X$ is topologically free.
If $I$ is a non-zero ideal of $C(X) \rtimes _r G$, then $I \cap C(X) \rtimes _r K \neq 0$ for every compact open subgroup $K<G$.
\end{lemma}

\begin{proof}
First, we will show that for every compact open subgroup $K < G$, $a \in p_K (C(X) \rtimes _r G) p_K$ and $\varepsilon > 0$, there exists $h \in C(X)^K$ such that
\begin{enumerate}
\item $0 \leq h \leq 1$,
\item $\| [h]_K E_K(a) [h]_K \| > \| E_K (a) \| - \varepsilon$,
\item $\| [h]_K E_K(a) [h]_K - [h]_K a [h]_K \| < \varepsilon$.
\end{enumerate}
We may assume that $a \in p_K (C_c(G , C(X))) p_K$.
Since the action is topologically free, there exists $x \in X$ such that $| a_e (x) | > \| a_e \| - \varepsilon$ and $G_x = \{ e \}$ where $a_e \in C(X)^K$ is a function such that $E_K(a) = [a_e]_K$.
Take a finite subset $F \subset G$ such that $a = \sum _{g \in F} a|_{KgK}$ and $F \ni g \mapsto KgK \in K \backslash G / K$ is injective.
For each $g \in F \setminus \{ e \}$, there is a $K$-invariant open neighborhood $U_g \subset X$ of $x$ such that $U_g \cap gU_g = \emptyset$.
Indeed, for a neighborhood system $\{ V_{\lambda} \}$ of $x \in X$, let $U_{\lambda} := KV_{\lambda}$.
Suppose that $U_{\lambda} \cap g U_{\lambda} \neq \emptyset$ for all $\lambda$, then there exist $k_{\lambda}, k'_{\lambda} \in K$ and $v_{\lambda}, v'_{\lambda} \in V_{\lambda}$ such that $k_{\lambda} v_{\lambda} = g k'_{\lambda} v'_{\lambda}$.
We may assume that $k_{\lambda}, k'_{\lambda}$ converges to some $k, k' \in K$ respectively, and then one has $kx = gk'x$.
This is a contradiction since $g \not\in K$ and $G_x = \{ e \}$.
Let $h_g \in C(X)^K = C(K \backslash X)$ be a function such that $0 \leq h_g \leq 1$, $h_g (x) = 1$, and $\supp h_g \subset U_g$ and let $h := \prod _{g \in F \setminus \{ e \}} h_g$.
Then $0 \leq h \leq 1$, $h(x)=1$, and 
\begin{align*}
\| [h]_K E_K(a) [h]_K \| = \| h a_e h \| \geq |a_e (x)| > \| E_K(a) \| - \varepsilon .
\end{align*}
For (3), by $a(ks) = p_Ka (ks) = k.(a (s))$ for $k \in K$ and $s \in G$, one has
\begin{align*}
\left[ [h]_K a |_{KgK} [h]_K \right] (kgl) &= k.(h \cdot g.h) \left( \int _{kgK} a|_{KgK} (s) \, ds \right) \\
&= 0
\end{align*}
for $g \in F \setminus \{ e \}$ and $k, l \in K$ and $\left[ [h]_K a |_{KgK} [h]_K \right] (s) = 0$ for $s \not\in KgK$.
Therefore, we have $[h]_K a [h]_K = [h]_K [a_e]_K [h]_K = [h]_K E_K(a) [h]_K$ and this $h$ has the desired property.

Suppose that there exists a non-zero ideal $I$ in $C(X) \rtimes _r G$ such that $I \cap C(X) \rtimes _r K = 0$ for some compact open subgroup $K<G$.
Take a non-zero positive element $a \in I$ and a compact open subgroup $L < G$ such that $p_L a p_L \neq 0$ and $L < K$.
Note that $E_L( p_L a p_L )$ is non-zero since $E_L$ is faithful.
Let $q \colon C(X) \rtimes _r G \to C(X) \rtimes _r G / I$ be the canonical quotient map.
Note that $q | _{C(X) \rtimes _r K}$ is isometry since $I \cap C(X) \rtimes _r K = 0$.
By the first paragraph of this proof, there exists $h \in C(X)^L$ such that 
\begin{enumerate}
\item $0 \leq h \leq 1$, 
\item $\displaystyle \| [h]_L E_L( p_L a p_L ) [h]_L \| > \frac{2}{3} \| E_L ( p_L a p_L ) \|$,
\item $\displaystyle \| [h]_L E_L( p_L a p_L ) [h]_L - [h]_L a [h]_L \| < \frac{1}{3} \| E_L (p_L a p_L) \|$.
\end{enumerate}
One has 
\begin{align*}
\frac{1}{3} \| E_L (p_L a p_L) \| &> \| q ([h]_L E_L( p_L a p_L ) [h]_L - [h]_L a [h]_L) \| \\
&= \| q ([h]_L E_L( p_L a p_L ) [h]_L) \| \\
&= \| [h]_L E_L( p_L a p_L ) [h]_L \| \\
&> \frac{2}{3} \| E_L ( p_L a p_L ) \|
\end{align*}
and this is a contradiction.
\end{proof}

\begin{proof}[Proof of Theorem \ref{main}]
Let $I$ be a non-zero ideal in $C(X) \rtimes _r G$.
By Proposition \ref{ideal} and Proposition \ref{intersection}, $I \cap C(X) \rtimes _r K = C_0 (U) \rtimes _r K$ for 
some non-empty $K$-invariant open subset $U \subset X$. 
Let $\displaystyle V = \bigcup_{g \in G} gU$.
Then $V$ is a non-empty $G$-invariant open subset and $C_0 (V) = \overline{\Span} \{ g.f \mid g \in G, f \in C_0 (U) \}$ by Stone--Weierstrass theorem.
We will prove that $C_0 (V) \rtimes _r G \subset I$.
It suffices to prove that $f p_L \in I$ for all $f \in C_0(V)$ and compact open subgroup $L < G$ since $C_0 (V) \cdot \mathrm{C}^*_{\lambda} (G) \subset C_0(V) \rtimes _r G$ is dense and $f p_L x \to fx$ as $L \to \{ e \}$ for $f \in C_0(V)$ and $x \in \mathrm{C}^*_{\lambda}(G)$.
We may assume that $f \in C_0(V)$ is of the form $f = \sum_{i=1}^n g_i . f_i$ where $g_i \in G$ and $f_i \in C_0(U)$.
There exists a compact open subgroup $L_0 < K$ such that $g_i^{-1} L_0 g_i < K$ for all $i=1, \ldots , n$.
For a compact open subgroup $L<L_0$, $f p_L = \sum _{i=1}^n u_{g_i} f_i p_{g_i^{-1} L g_i} u_{g_i}^* \in I$ since $f_i p_{g_i^{-1} L g_i} \in C_0(U) \rtimes _r K \subset I$.
Thus, $f p_L \in I$ for all compact open subgroup $L < G$ since $f p_L = f p_M p_L \in I$ for $M < L \cap L_0$.
Therefore, we have $C_0 (V) \rtimes _r G \subset I$.
\end{proof}

\begin{corollary}\label{maincor}
Let $G$ be a totally disconnected locally compact group and $X$ be a compact $G$-space.
Assume that the action $G \curvearrowright X$ is minimal, topologically free, and there exists a compact open subgroup $K$ of $G$ such that the restricted action $K \curvearrowright X$ is free.
Then the reduced crossed product $C(X) \rtimes _r G$ is simple.

In particular, if a totally disconnected locally compact group $G$ admits a topologically free boundary, then the reduced crossed product $C(\partial _F G) \rtimes _r G$ is simple.
\end{corollary}

\begin{proof}
By the result of Le Boudec and Tsankov \cite[Corollary 1.3.]{LeBoudecTsankov}, if a locally compact group admits a topologically free boundary, then its action on its Furstenberg boundary is free.
Hence, the corollary follows from Theorem \ref{main}.
\end{proof}

In Theorem \ref{main}, the assumption that there exists a compact open subgroup of $G$ to which the restricted action is free is necessary; see the example below.
The author is grateful to Narutaka Ozawa for sharing this example.

\begin{example}\label{exampletopfree}
There is a topologically free action of a totally disconnected compact abelian group on a compact Hausdorff space that does not have the intersection property.

Let $\Gamma \curvearrowright A$ be a free action of a finite abelian group on an infinite set; for instance, $\mathbb{Z} / 2 \mathbb{Z} \curvearrowright (\mathbb{Z} / 2 \mathbb{Z}) ^{\mathbb{N}}$ by multiplication. 
Let $T$ be a rooted tree such that each vertex has $A$ descendants and $X$ be a compactification $\overline{T}$ of $T$,i.e., the disjoint union of $T$ and the set of ends $\partial T$ of $T$, thus, as a set, $\displaystyle X := \bigcup _{d \in \{ 0 \} \cup \mathbb{N} \cup \{ \infty \}} A ^d$.
This set $X$ is topologized by declaring that sets of the form 
\[
\{ x \} \cup \{ y \in \overline{T} \mid [x , y] \cap F = \emptyset \}
\]
where $x \in \overline{T}$, $F \subset T$ is a finite subset, and $[x, y]$ is a unique geodesic connecting $x$ with $y$, are open (see \cite[Section 5.2.]{BrownOzawa}).
Viewing elements in $X$ as finite or infinite sequences in $A$, $\ell (x) \in \{ 0 \} \cup \mathbb{N} \cup \{ \infty \}$ denotes the word length of $x \in X$.
Then one has $x_n \to x$ in $X$ if and only if 
\begin{itemize}
\item for each $k \leq \ell (x)$, there is $N \in \mathbb{N}$ such that for all $n \geq N$ one has $\ell (x_n) \geq k$ and $x_n (k) = x (k)$,
\item for $k = \ell (x) +1 \in \mathbb{N}$, and for any finite set $E \subset A$, there is $N \in \mathbb{N}$ such that for all $n \geq N$ one has either $\ell (x_n) < k$ or $x_n (k) \not\in E$.
\end{itemize}
The second condition is vacuous if $\ell (x) = \infty$.

The compact abelian group $K := \Gamma ^{\mathbb{N}}$ acts on $X$ by the pointwise multiplication.
This action is topologically free since $K_x = \{ e \}$ for every $x \in X$ with $\ell (x) = \infty$ and $\{ x \in X \mid \ell (x) = \infty \}$ is dense in $X$.
For $d \in \{ 0 \} \cup \mathbb{N}$, we write $X_d := \{ x \in X \mid \ell (x) \leq d \} \subset X$ and $K = K_d \times K_d'$ where $K_d = \Gamma ^{\{ 1, 2, \ldots , d \}}$ and $K_d' = \Gamma ^{\{ d+1, d+2, \ldots  \}}$.
Note that $X_d$ is closed and $K$-invariant, and $K_d'$ acts on $X_d$ trivially.
Hence, one obtains the corresponding $*$-homomorphism
\[
\pi _d \colon C(X) \rtimes K \to C(X_d) \rtimes K \cong (C(X_d) \rtimes K_d) \otimes \mathrm{C}^* (K_d').
\]
Then $\pi _d (p_K) = p_{K_d} \otimes p_{K_d'}$.
Note that $p_{K_d'}$ is the minimal projection in $\mathrm{C}^* (K_d') \cong c_0(\hat{K_d'})$ that corresponds to the unit character.
We take a non-trivial character $\chi _d$ on $K_d'$.
Then one has $\chi _d (p_{K_d'}) = 0$ since a character $\chi$ on $\mathrm{C}^*(K_d')$ with $\chi (p_{K_d'}) = 1$ should be equal to the unit character.
We take a compact open normal subgroup $L_d \lhd K_d'$ such that $\chi _d (p_{L_d}) = 1$.
Such a compact open normal subgroup exists since $(p_L)_{L \lhd K_d' : \text{compact open normal subgroup}}$ is an approximate unit for $\mathrm{C}^*(K_d')$.
Let 
\[
I := \bigcap _d \ker{((\mathrm{id} \otimes \chi _d) \circ \pi _d )} \lhd C(X) \rtimes K.
\]
This ideal is non-zero since $p_K \in I$.
We claim that there is no non-zero $K$-invariant ideal $J \lhd C(X)$ such that $J \rtimes K \subset I$.
Suppose that such an ideal $J$ exists.
Take $h \in J$ and $x \in X$ such that $h(x) \neq 0$.
We may assume that $d := \ell (x) < \infty$ since $\{ x \in X \mid \ell (x) < \infty \}$ is dense in $X$.
Then
\[
((\mathrm{id} \otimes \chi _d) \circ \pi _d )(hp_{L_d}) = h |_{X_d} \in C(X_d) \subset C(X_d) \rtimes K_d
\]
is non-zero.
This is a contradiction.
\end{example}

We apply Corollary \ref{maincor} to two examples.

\begin{corollary}
Let $\mathcal{N}_{d,k}$ be the Neretin group.
Then $C( \partial _F \mathcal{N}_{d,k} ) \rtimes _r \mathcal{N}_{d,k}$ is simple.
\end{corollary}

\begin{proof}
In \cite[Theorem 1.3.]{CapraceLeBoudecMatteBon}, Caprace, Le Boudec, and Matte Bon proved that any totally disconnected locally compact group with a piecewise minimal-strongly-proximal action on a totally disconnected compact space admits a topologically free boundary.
Since the action $\mathcal{N}_{d,k} \curvearrowright \partial T_{d,k}$ is piecewise minimal-strongly-proximal, the Neretin group admits a topologically free boundary.
Therefore, by Corollary \ref{maincor}, $C( \partial _F \mathcal{N}_{d,k} ) \rtimes _r \mathcal{N}_{d,k}$ is simple.
\end{proof}

\begin{proposition}
Let $G$ be an HNN extension corresponding to a totally disconnected locally compact group $H$, its open subgroup $H_0$, and an open, injective, continuous homomorphism $\theta \colon H_0 \hookrightarrow H$, and $T$ be its Bass--Serre tree.
Assume that $1 < [H : H_0] < \infty $, $1 < [H : \theta (H_0)] < \infty$, there exist $a \in H \setminus H_0$ and $b \in H \setminus \theta (H_0)$ such that $a$ centralizes $H_0$ and $b$ centralizes $\theta (H_0)$, and $\bigcap _{n \geq 0} \theta ^{-n} (H_0^{(n)})$ is trivial where $H_0^{(0)} = H_0$ and $H_0^{(n+1)} = H_0 \cap \theta (H_0^{(n)})$.
Then the action of $G$ on the boundary of the Bass--Serre tree $\partial T$ is topologically free.

In particular, the action $\mathrm{G} (m,n) \curvearrowright \partial T_{|m|+n}$ is topologically free.
\end{proposition}

\begin{proof}
Consider an element $\xi$ of $\partial  T$ represented by a ray 
\[
btH, btat^{-1}H, btat^{-2}H, btat^{-3}H, btat^{-3}btH, btat^{-1}H, btat^{-2}H, \ldots
\]
(this element is represented by $\prod _{i=1}^{\infty} (btat^{-1})(t^{-2})^i$).
We will show that $G_{\xi}$ is trivial.
Take $g \in G_{\xi}$ arbitrarily.
Then $g$ fixes some ray representing $\xi$ or acts on some ray representing $\xi$ by translation.
Assume that $g$ acts on some ray representing $\xi$ by translation of degree $d \geq 1$.
Then for sufficiently large $N$, writing $\prod_{i= 1}^N (btat^{-1})(t^{-2})^i$ as $x$, $g$ sends the three vertices $xH, xbtH, xbtat^{-1}H$ to $xbtat^{-(d-1)}H, xbtat^{-d}H, xbtat^{-(d+1)}H$.
However, such an element $g$ does not exist because $G$-action on the tree $T$ preserves directions of edges, which is defined as follows; the source of an edge $gH_0$ is $gH$ and the range of $gH_0$ is $gt^{-1}H$.
Thus, $g$ fixes some ray representing $\xi$.
In this case, for large $N$, writing $\prod_{i= 1}^N (btat^{-1})(t^{-2})^i$ as $x$, 
\[
g \in x \left( H \cap \bigcap _{i \geq N+1} \left( \prod _{j=N+1}^i (btat^{-1}) (t^{-2})^j \right) H \left( \prod _{j=N+1}^i (btat^{-1}) (t^{-2})^j \right) ^{-1} \right) x^{-1}.
\]
It is easy to see that 
\[
H \cap \left( \prod _{j=N+1}^i (btat^{-1}) (t^{-2})^j \right) H \left( \prod _{j=N+1}^i (btat^{-1}) (t^{-2})^j \right) ^{-1} \subset \theta ^{-M(i)} (H_0^{(M(i)+1)})
\]
for $i \geq N+1$ where $M(i)$ is a some positive integer with $M(i) \to \infty$ as $i \to \infty$.
Therefore, $g$ is trivial.
\end{proof}


We now proceed to prove the partial converse to Corollary \ref{maincor}.

\begin{theorem}\label{converse}
Let $G$ be a totally disconnected locally compact group and $X$ be a compact $G$-space.
Assume that the action $G \curvearrowright X$ is amenable and the stabilizer map $X \ni x \mapsto G_x \in \Sub{G}$ is continuous.
If $C(X) \rtimes _r G$ is simple, then the action $G \curvearrowright X$ is free.
\end{theorem}

In the following, we fix a family of measures $\{ \mu _y \} _{y \in X}$ in Lemma \ref{measure}.

\begin{proposition}\label{function}
Let $G$ be a totally disconnected locally compact group and $X$ be a compact $G$-space.
Assume that the stabilizer map $X \ni x \mapsto G_x \in \Sub{G}$ is continuous.
If the action $G \curvearrowright X$ is not free, then there is a non-zero function $f_0 \in C_c (G, C(X))$ such that $\displaystyle \int _{G_y} [f_0 (sh)](y) \Delta _G (h) ^{\frac{1}{2}} \Delta _{G_y} (h)^{- \frac{1}{2}} \, d \mu_{y}(h) =0$ for every $y \in X$ and $s \in G$ where $\Delta _G , \Delta _{G_y}$ are the modular functions of $G, G_y$ respectively.
\end{proposition}

\begin{proof}
Since the action $G \curvearrowright X$ is not free, there exists $x \in X$ such that $G_x$ is non-trivial.
Take two distinct points $g_1, g_2 \in G_x$ and two positive functions $f_1, f_2 \in C_c (G)$ such that $\supp{f_1} \cap \supp{f_2} = \emptyset$ and $f_i (g_i) > 0$ for $i = 1, 2$.
Define $\eta _0$ by
\[
[\eta _0 (s)] (y) = f_1 (s) \int _{G_y} f_2 (h) \, d \mu _y (h) - f_2 (s)  \int _{G_y} f_1 (h) \, d \mu _y (h).
\]
Then $\eta _0 \in C_c (G, C(X))$, $[\eta _0 (\cdot)](x)| _{G_x} \neq 0$, and $\int _{G_y} [\eta_0 (h)] (y) \, d \mu _y (h) = 0$ for every $y \in X$.
Define a map $\tau \colon \{ (g, y) \in G \times X \mid g.y=y \} \to \mathbb{R}$ by $\tau (g, y) := \Delta _{G_y}(g) ^{-\frac{1}{2}}$.
This map is continuous.
Indeed, fix some compact open subgroup $K$ of $G$ and assume $(g_i, y_i) \to (g, y)$ in $\{ (g, y) \in G \times X \mid g.y=y \}$.
Then
\begin{align*}
\mu _{y_i} (K \cap G_{y_i}) = \int _{G_{y_i}} \chi _K (h) \, d \mu _{y_i} (h) \to \int _{G_y} \chi _K (h) \, d \mu _y (h) = \mu _y (K \cap G_y)
\end{align*}
and
\begin{align*}
| \mu _{y_i} (K g_i \cap G_{y_i}) - \mu _y (Kg \cap G_y)| &= \left| \int _{G_{y_i}} \chi _{K{g_i}} (h) \, d \mu _{y_i} (h) - \int _{G_y} \chi _{Kg} (h) \, d \mu _y (h) \right| \\
&\leq \left| \int _{G_{y_i}} \chi _{K{g_i}} (h) \, d \mu _{y_i} (h) - \int _{G_{y_i}} \chi _{K g} (h) \, d \mu _{y_i} (h) \right| \\
&\  + \left| \int _{G_{y_i}} \chi _{K g} (h) \, d \mu _{y_i} (h) - \int _{G_y} \chi _{Kg} (h) \, d \mu _y (h) \right|
\end{align*}
hold.
Since $g_i \in Kg$ eventually, $K g_i = K g$ eventually and one has 
\begin{align*}
\left| \int _{G_{y_i}} \chi _{K{g_i}} (h) \, d \mu _{y_i} (h) - \int _{G_{y_i}} \chi _{K g} (h) \, d \mu _{y_i} (h) \right| \to 0.
\end{align*}
Moreover, by the definition of $\{ \mu _y \} _{y \in X}$, one has
\[
\left| \int _{G_{y_i}} \chi _{K g} (h) \, d \mu _{y_i} (h) - \int _{G_y} \chi _{Kg} (h) \, d \mu _y (h) \right| \to 0.
\]
Therefore, we have
\[
\tau (g_i, y_i) = \left( \frac{\mu _{y_i} (K \cap G_{y_i})}{\mu _{y_i} (K g_i \cap G_{y_i})} \right) ^{\frac{1}{2}} \to \left( \frac{\mu _y (K \cap G_y)}{\mu _y (Kg \cap G_y)} \right) ^{\frac{1}{2}} = \tau (g,y).
\]
Let $\tilde{\tau} \in C(G \times X)$ be an extension of $\tau$ whose existence is ensured by the Tietze extension theorem.
Define $\eta$ by $[\eta (h)] (y) := \tilde{\tau} (h,y) \Delta _G (h)^{-\frac{1}{2}} [\eta _0 (h)] (y)$.
Then $\eta \in C_c (G, C(X))$ and $[\eta (\cdot)](x)| _{G_x} \neq 0$.
By the Tietze extension theorem, one can find a non-zero function $\zeta \in C_c (G)$ such that 
\[
\int _{G_x} \zeta (k^{-1}) [\eta (k)] (x) \, d \mu _x (k) \neq 0.
\]
Define $f_0$ by $\displaystyle [f_0 (h)](y) := \int _{G_y} \zeta (hk^{-1}) [\eta (k)](y) \, d \mu _y (k)$.
Then $f_0 \in C_c (G, C(X))$ and $f_0 \neq 0$ since $[f_0(e)] (x) \neq 0$.
For $y \in X$ and $s \in G$, 
\begin{align*}
&\int _{G_y} [f_0 (sh)](y) \Delta _G (h) ^{\frac{1}{2}} \Delta _{G_y} (h)^{- \frac{1}{2}} \, d \mu_{y}(h) \\
&= \int _{G_y} \int _{G_y} \zeta (shk^{-1}) [\eta (k)] (y) \Delta _G (h) ^{\frac{1}{2}} \Delta _{G_y} (h)^{- \frac{1}{2}} \, d \mu _y (k) \, d \mu _y (h) \\
&= \int _{G_y} \left( \int _{G_y} [ \eta _0 (l h) ](y) \, d \mu _y (h) \right) \zeta (sl^{-1}) \Delta _{G}(l)^{- \frac{1}{2}} \Delta _{G_y}(l) ^{-\frac{1}{2}} \, d \mu _y (l) \\
&= 0
\end{align*}
holds and this $f_0$ has the desired property.
\end{proof}

\begin{proof}[Proof of Theorem \ref{converse}]
Suppose that the action $G \curvearrowright X$ is not free.
Fix some $x \in X$ such that $G_x$ is non-trivial.
Let $\rho$ be a rho-function for $(G, G_x)$ and $\nu$ be the associated quasi-invariant regular Borel measure on $G/G_x$; they satisfy $\displaystyle \frac{d g. \nu}{d \nu} (a G_x) = \frac{\rho (g^{-1}a)}{\rho (a)}$.
Define a $*$-homomorphism $\pi \colon C(X) \to B(L^2( G/G_x ))$ and a group homomorphism $u \colon G \to \mathcal{U} (L^2(G/G_x))$ by the relations
\begin{align*} 
&[\pi (f) \xi](a G_x) := f(a.x) \xi (a G_x), \\
&[u_g \xi](a G_x) := \left( \displaystyle \frac{d g.\nu}{d \nu} (a G_x) \right) ^{\frac{1}{2}} \xi (g^{-1}a G_x).
\end{align*}
Then this is a covariant representation and we obtain a $*$-homomorphism $\tilde{\pi} \colon C(X) \rtimes _r G \to B(L^2 (G / G_x))$ from $\pi$ and $u$ since the action $G \curvearrowright X$ is amenable.
We will show that $\tilde{\pi}$ is not faithful.

By Proposition \ref{function}, there is a non-zero function $f_0 \in C_c (G, C(X))$ such that $\displaystyle \int _{G_y} [f_0 (sh)](y) \Delta _G (h) ^{\frac{1}{2}} \Delta _{G_y} (h)^{- \frac{1}{2}} \, d \mu_{y}(h) =0$ for every $y \in X$ and $s \in G$.
Let $[f(s)](y) = [f_0 (s)](s^{-1}.y)$ and we will show that $\tilde{\pi} (f) = 0$.
Take $\xi \in C_c (G/G_x) \subset L^2 (G/G_x)$ arbitrarily.
Then for $aG_x \in G/G_x$, one has
\begin{align*}
&[\tilde{\pi} (f) \xi] (a G_x) \\ 
&= \int _G [f(g)](a.x) \left( \displaystyle \frac{d g.\nu}{d \nu} (a G_x) \right) ^{\frac{1}{2}} \xi (g^{-1}a G_x) \, d\mu (g) \\
&= \int _G [f(ag^{-1})](a.x) \left( \frac{\rho (g)}{ \rho (a)} \right) ^{\frac{1}{2}} \Delta _G (g)^{-1} \xi (g G_x) \, d \mu (g) \\
&= \int _{G/G_x} \int _{G_x} [f(ah^{-1}b^{-1})](a.x) \left( \frac{\rho (bh)}{ \rho (a)} \right) ^{\frac{1}{2}} \Delta _G (bh)^{-1} \xi (b G_x) \rho (bh)^{-1} \, d \mu _x (h) \, d \nu (b G_x) \\
&= \int _{G/G_x} \left( \int _{G_x} [f_0 (ah^{-1}b^{-1})](b.x) (\Delta _G (h) \Delta _{G_x} (h)) ^{- \frac{1}{2}} \, d \mu _x (h) \right) (\rho (a) \rho (b))^{-\frac{1}{2}} \Delta _G (b)^{-1} \xi (b G_x) \, d \nu (b G_x)
\end{align*}
where the third equality uses the following formula; for $\phi \in C_c (G)$,
\[
\displaystyle \int _G \varphi (g) \, dg = \int _{G/G_x} \int _{G_x} \varphi (bh) \rho (bh)^{-1} \, d \mu _x (h) \, d \nu (bG_x)
\]
and the fourth equality uses the following formula; for $b \in G$ and $h \in G_x$, 
\[
\rho(bh) = \frac{\Delta _{G_x} (h)}{\Delta _G (h)} \rho (b).
\]
For each $b \in G$,
\begin{align*}
&\int _{G_x}  [f_0(ah^{-1}b^{-1})](b.x) (\Delta _G (h) \Delta _{G_x} (h)) ^{- \frac{1}{2}} \, d \mu _x (h) \\
&= \int _{G_x} [f_0(ahb^{-1})](b.x) \Delta _G (h)^{\frac{1}{2}} \Delta _{G_x} (h) ^{- \frac{1}{2}} \, d \mu _x (h) \\
&= c_b \int _{G_{b.x}}  [f_0(ab^{-1}k)](b.x) \Delta _G (k) ^{\frac{1}{2}} \Delta _{G_{b.x}} (k) ^{- \frac{1}{2}} \, d \mu _{b.x} (k) \\
&= 0
\end{align*}
holds where $c_b$ is defined as follows; for $b \in G$, $(\Ad{b})_{*} \mu _x$ is a left Haar measure on $G_{b.x}$ where $\Ad{b} \colon G_x \to G_{b.x}$ is the conjugation.
A positive number $c_b > 0$ is defined as satisfying $(\Ad{b})_{*} \mu _x = c_b \mu _{b.x}$.
Therefore, $\tilde{\pi}(f) = 0$ and $C(X) \rtimes _r G$ is not simple.
\end{proof}

\begin{corollary}\label{conversecor}
Let $G$ be an exact totally disconnected locally compact group.
If $C( \partial _F G ) \rtimes _r G$ is simple, then the action $G \curvearrowright \partial _F G$ is free.
\end{corollary}

\begin{proof}
Since $G$ is exact, the action $G \curvearrowright \partial _F G$ is amenable by Corollary \ref{Furstenbergamenable}.
By \cite[Corollary 1.3.]{LeBoudecTsankov}, the stabilizer map associated with the Furstenberg boundary is continuous.
Therefore, the simplicity of the crossed product $C(\partial _F G) \rtimes _r G$ implies the freeness of the action $G \curvearrowright \partial _F G$ by Theorem \ref{converse}.
\end{proof}


\end{document}